\documentclass[10pt]{amsart}
\usepackage{amsmath,amssymb,amscd,mathrsfs,amsthm,amscd,inputenc,enumerate,amsfonts,bbold,pdfsync}
\usepackage[english]{babel}
\usepackage{color}
\usepackage{graphicx}
\usepackage{mathtools}
\usepackage{verbatim}
\usepackage[shortlabels]{enumitem}
\newcommand{\spec}{{\rm Spec}}
\newcommand{\Zar}{{\rm Zar}}

\newcommand{\Kr}{{\rm Kr}}
\newcommand{\Max}{{\rm Max}}

\newcommand{\qspec}{{\rm QSpec}}
\newcommand{\qmax}{{\rm QMax}}

\newcommand{\ms}{\mathscr}

\newcommand{\SStar}{{\mbox{\rm\texttt{SStar}}}} 
\newcommand{\SStarf}{\mbox{\rm\texttt{SStar}}_{\!\mbox{\tiny\it\texttt{f}}}} 
\newcommand{\SStarb}{\mbox{\rm\texttt{SStar}}_{\!\mbox{\tiny\it\texttt{b}}}} 

\newcommand{\SStarstab}{\overline{\mbox{\rm\texttt{SStar}}}}

\newcommand{\SStarsp}{\mbox{\rm\texttt{SStar}}_{\!\mbox{\tiny\it\texttt{sp}}}}
\newcommand{\SStarspf}{\mbox{\rm\texttt{SStar}}_{\!\mbox{\tiny\it\texttt{f,sp}}}}
\newcommand{\SStarstabft}{\SStarspf}

\newcommand{\SStareabf}{\mbox{\rm\texttt{SStar}}_{\!\mbox{\tiny\it\texttt{f,eab}}}}
\newcommand{\SStareab}{\mbox{\rm\texttt{SStar}}_{\!\mbox{\tiny\it\texttt{eab}}}}
\newcommand{\SStarval}{\mbox{\rm\texttt{SStar}}_{\!\mbox{\tiny\it\texttt{val}}}}

\newcommand{\locsist}{\mbox{\rm\texttt{LS}}}
\newcommand{\locsistf}{\mbox{\rm\texttt{LS}}_{\!\mbox{\tiny\it\texttt{f}}}}
    
\newcommand{\stf} {\star{_{\!{_f}}}} 
\newcommand{\stt} {\widetilde{\star}} 
\newcommand{\stu} {\overline{\star}} 
\newcommand{\sta} {\star_{\mbox{\it\tiny\texttt{a}}}}

\newcommand{\cont}{\boldsymbol{c}}

\newcommand{\overr}{{\mbox{\rm\texttt{Overr}}}}
\newcommand{\overric}{{\mbox{\rm\texttt{Overr}}}_{\!\mbox{\tiny\it\texttt{ic}}}}

 \newcommand{\FF}{\boldsymbol{\overline{F}}}
    \newcommand{\F}{\boldsymbol{F}}
    \newcommand{\f}{\boldsymbol{f}}
    \newcommand{\insid}{\boldsymbol{\mathcal{I}}}
 \DeclareMathOperator{\chius}{\mbox{\rm\texttt{Cl}}}

  \newcommand{\insZ}{\mathbb Z}
 \newcommand{\insQ}{\mathbb Q} 
 \newcommand{\insA}{\mathbb A}
 
\newtheoremstyle{mio}%
	{}{} 
	{\itshape}{} 
	{\bfseries}{.}{ } 
	{#1 #2\thmnote{\mdseries~(\scshape #3)}} 
\theoremstyle{mio}
\newtheorem{teor}{Theorem}[section]
\newtheorem{cor}[teor]{Corollary}
\newtheorem{prop}[teor]{Proposition}
\newtheorem{lemma}[teor]{Lemma}

\theoremstyle{definition}

\newtheorem{ex}[teor]{ \textbf{Example}}
\newtheorem{oss}[teor]{Remark}

\DeclareMathOperator{\Cl}{\mbox{\rm\texttt{Cl}}}

\newcommand{\mathttt}[1]{\mbox{\tiny{\texttt{#1}}}}

 \newcommand{\X}{\mathbb{X} } 
\newcommand{\bmap}{\boldsymbol{\beta}}

\begin{document}

\title{Spectral spaces of semistar operations}

\author{Carmelo A. Finocchiaro}
\email{carmelo@mat.uniroma3.it}

\author{Marco Fontana}
\email{fontana@mat.uniroma3.it}

\author{Dario Spirito}
\email{spirito@mat.uniroma3.it}

\address{Dipartimento di Matematica e Fisica, Universit\`a degli Studi
``Roma Tre'', Roma, Italy}

\keywords{Spectral spaces, Riemann-Zariski space of valuation domains, semistar operations, Zariski topology, inverse topology, ultrafilter topology, localizing system}
\subjclass[2010]{13A15, 13G05, 13B10, 13E99, 13C11, 14A05}
\thanks{This work was partially supported by {\sl GNSAGA} of {\sl Istituto Nazionale di Alta Matematica}.}

\maketitle

\begin{abstract}
We investigate, from a topological point of view, the classes of spectral semistar operations and of \texttt{eab} semistar operations, following me\-thods recently introduced  in  \cite{Fi, FiSp}. We show that, in both cases, the subspaces of finite type operations are spectral spaces in the sense of Hochster and, moreover, that   there is a distinguished class of overrings strictly connected to each of the two types of collections of semistar operations.
We also prove that the space of  stable semistar operations is homeomorphic to the space of Gabriel-Popescu localizing systems, endowed with a Zariski-like topology, extending to   the  topological level a result established in \cite{fohu}.
   As a side effect, we obtain that the space of localizing systems of finite type is also a spectral space. Finally, we show that the Zariski topology on the set of semistar operations is the same as the $b$-topology defined recently by B. Olberding  \cite{ol, olb_noeth}. 
\end{abstract}


\section{Introduction}  
 In 1936, W. Krull introduced,   in his first {\sl Beitr\"age} paper \cite{Krull:1936} (see also \cite{Krull}), the concept of a ``special'' closure operation on the nonzero fractional ideals, called star operation. If $D$ is an integral domain with quotient field $K$, in 1994, Okabe and 
Matsuda \cite{OM2} suggested the terminology of semistar operation for a more 
``flexible'' and general notion of a closure operation $ \star $, defined on the set  of nonzero $D$-submodules of $K$, allowing $  D \neq D^\star$. However, it is worth noting that this kind of operation was previously considered by J. Huckaba, in the very general setting of rings with zero divisors \cite[Section 20]{hu} (cf. also \cite[Section 32]{gi}, \cite{an-overrings}, \cite{ac}, \cite{an}, \cite{ep-12}, \cite{ep-15}, \cite{fo-lo-2003}, \cite{hk-98}, \cite{hk-00}, \cite{hk-01}, \cite{hk-11}).

The set of semistar operations on a domain $D$ can be endowed with a topology (called the \emph{Zariski topology}), as in \cite{FiSp}, in such a way that both the prime spectrum of $D$ and the set of overrings of $D$ are naturally  topologically embedded in it.
 This topology was used to study the problem of when the semistar operation defined by a family of overrings is of finite type \cite[Problem 44]{ch-gl}. Subsequently, it was proved that the set of semistar operations of finite type is a spectral space.

The purpose of this paper is to deepen and specialize the study of the Zariski topology on $\SStar(D)$ to the case of the distinguished subspaces $\SStarstab(D)$ (Section \ref{sect:stable} and \ref{sect:spectral}) and $\SStareab(D)$ (Section \ref{sect:eab}) comprising, respectively, the stable and the \texttt{eab} semistar operations.
  We will show that, in both cases, there is a topological retraction to the set of finite type operations and that there is a distinguished class of overrings connected to each of the two types of collections of semistar operations.
  We will also show that both the set of finite type spectral operations and the set of finite type \texttt{eab} operations are spectral spaces, reducing the latter case to the former. 
  However, the proofs given here are not constructive, and provide only vague hints on how such a ring might look like.
  We also prove that the space of  stable semistar operations is homeomorphic to the space of Gabriel-Popescu localizing systems, endowed with a natural topology (described later), extending to  the topological level a result established in \cite{fohu}.
   As a side effect, we obtain that the space of localizing systems of finite type is also a spectral space. Finally, we show that the Zariski topology on the set of semistar operations is the same as the $b$-topology defined by Olberding \cite{ol, olb_noeth}.
   
\section{Preliminaries}

Throughout this paper, let $D$ be an integral domain with quotient field $K$. Let $\FF(D)$ [respectively, $\F(D)$; $\f(D)$] be the set of all nonzero $D$--submodules of $K$ [respectively, nonzero fractional ideals; nonzero finitely 
generated fractional ideals] of $D$ (thus, $\f(D)\subseteq\F(D)\subseteq\FF(D)$).

A mapping $\star:\FF(D)\longrightarrow\FF(D)$, $E\mapsto E^\star$, is called a \emph{semistar operation} of $D$ if, for all $z\in K$, $z\neq 0$ and for all $E,F \in\FF(D)$, the following properties hold: $\mathbf{\bf(\star_1)} \;(zE)^\star =zE^\star$; $\mathbf{\bf (\star_2)} \; E\subseteq F \Rightarrow E^\star \subseteq F^\star$; $\mathbf{ \bf (\star_3)}\; E \subseteq E^\star$; and $\mathbf{ \bf (\star_4)}\;  E^{\star \star} := (E^\star)^\star = E^\star $.

When $D^\star=D$, the restriction of $\star$ to $\boldsymbol{F}(D)$ is called a
\emph{star operation} (see \cite[Section 32]{gi} for more details).

As in the classical star-operation setting, we associate to a 
semistar operation $\star$ of $D$ a new semistar operation
$\stf$ of $D$ defining, for every $E\in\FF(D)$, 
\begin{equation*}
E^{\stf} := \bigcup \{F^\star \mid  F \subseteq E,  F \in \f(D)\}.
\end{equation*}
We call $\stf$ the semistar operation of finite type of $D$ \emph{associated }to $\star$. If $\star=\stf$, we say that $\star$ is a 
semistar operation \emph{of finite type}  on $D$. Note that $(\stf)_{_{\! f}} = \stf$, so $ \stf$ is a semistar operation
of finite type   on  $D$.

\smallskip

We denote by $\SStar(D)$ [respectively, $\SStarf(D)$] the set of all semistar ope\-ra\-tions [respectively, semistar operations of finite type] on $D$.
Given two semistar operations $\star'$ and $\star''$ of $D$, we say that $\star' {\boldsymbol{\preceq} }\ \star''$ if $E^{\star'} \subseteq E^{\star''}$, for all $E \in\FF(D)$. The relation ``${\boldsymbol{\preceq}}$'' introduces a partial ordering in $\SStar(D)$. From the definition of $\stf$, we deduce that $\stf {\boldsymbol{\preceq}}\star$ and that $\stf$ is the largest semistar operation of finite type smaller or equal to $\star$.

Let $\mathbf{\mathscr{S}}$ be a nonempty set of semistar operations on $D$. For each $E \in \FF(D)$, define $\wedge_{\mathbf{\mathscr{S}}}$ as follows:
\begin{equation*}
E^{\wedge_{\mathbf{\mathscr{S}}} } := \bigcap \{ E^\star \mid \star \in \mathbf{\mathscr{S}}\}.
\end{equation*}
It is easy to see that $\wedge_{\mathbf{\mathscr{S}}}$ is a semistar operation on $D$ and it is the infimum of ${\mathbf{\mathscr{S}}}$ in the partially ordered set $(\SStar(D), {\boldsymbol{\preceq}})$. The semistar operation $\vee_{\mathbf{\mathscr{S}}}:= \bigwedge \{ \sigma \in \SStar(D) \mid \star \boldsymbol{\preceq} \sigma \mbox{ for all } \star \in \mathbf{\mathscr{S}} \}$  is the supremum of ${\mathbf{\mathscr{S}}}$ in $(\SStar(D), {\boldsymbol{\preceq}})$.
 
A nonzero ideal $I$ of $D$ is called a \emph{quasi-$\star$-ideal} if $I = I^\star \cap D$. A \emph{quasi-$\star$-prime} is a quasi-$\star$-ideal which is also a prime ideal. The set of all quasi-$\star$-prime ideals of $D$ is denoted by $ \qspec^\star(D)$. The set of maximal elements in the set of proper quasi-$\star$-ideals of $D$ (ordered by set-theoretic inclusion) is denoted by $\qmax^\star(D)$ and it is a subset of $\qspec^\star(D)$. It is well known that if $\star$ is a semistar operation of finite type then $\qmax^\star(D)$ is nonempty \cite[Lemma 2.3(1)]{fo-lo-2003}. A semistar operation $\star $ is called \emph{quasi-spectral} (or \emph{semifinite}) if each quasi-$\star$-ideal is contained in a quasi-$\star$-prime.

In \cite{FiSp}, the set $\SStar(D)$ of all semistar operation was endowed with a topology (called the \emph{Zariski topology}) having, as a subbasis of open sets, the sets of the type $\texttt{V}_E:=\{\star\in\SStar(D)\mid 1\in E^\star\}$, where $E$ is a nonzero $D$-submodule of $K$. This topology makes $\SStar(D)$ into a quasi-compact T$_0$ space.

For each overring $T$ of $D$, we can define a semistar operation of finite type $\wedge_{\{T\}}: \FF(D) \rightarrow \FF(D)$ by setting $E^{\wedge_{\{T\}}} : = ET$, for each $E\in\FF(D)$. If we have a whole family $\boldsymbol{\mathcal{T}}$ of overrings of $D$, we can consider the semistar operation  $\wedge_{\boldsymbol{\mathcal{T}}} := \bigwedge\{ \wedge_{\{T\}} \mid T \in \boldsymbol{\mathcal{T}}\}$.

For future reference, we state an embedding property that will be used later several times.

\begin{prop}\label{immersione}
Let $\overr(D)$ be the set of all overrings of $D$, endowed with the topology whose   basic open sets are of the form $\overr(D[x_1, x_2,\ldots,x_n])$,  for $x_1, x_2, \ldots,x_n$  varying in $ K$.
\begin{enumerate}[\rm(1)]
\item The map $\iota:\overr(D)\rightarrow\SStarf(D)$, defined by $\iota(T) := \wedge_{\{T\}}$, for each $T \in \overr(D)$, is a topological embedding  \cite[Proposition 2.5]{FiSp}.
\item The map $\pi: \SStarf(D)\rightarrow\overr(D)$, defined by $\pi(\star):=D^\star$,  for any $\star\in \SStarf(D)$, is a continuous surjection.
\item $\pi\circ\iota$ is the identity map of $\overr(D)$, that is, $\pi$ is a topological retraction. 
\end{enumerate}
\end{prop}
\begin{proof}
Part (1) is \cite[Proposition 2.5]{FiSp}, while  parts (2) and (3) follow from the fact that 
\begin{equation*}
\pi^{-1}(\overr(D[x]))=\SStarf(D)\cap \texttt{V}_{x^{-1}D}.\qedhere
\end{equation*}
\end{proof}

\begin{oss} Note that, by \cite[Proposition 2.4(2)]{FiSp}, a statement completely analogous   to Proposition \ref{immersione} holds when  $\SStar(D)$ replaces everywhere $\SStarf(D)$.
\end{oss}

Recall that two different points $x, y$ of a topological space are \emph{topologically distinguishable} if there is an open set which contains one of these points and not the other. Obviously, the previous property holds for any pair of distinct points if and only if the space is $T_0$. On the other hand,  ``topological indistinguishability'' of points is an equivalence relation. No matter what topological space $X$ might be to begin with, the quotient space under this equivalence relation is always $T_0$. This quotient space is called {\it the Kolmogoroff quotient space of} $X$.
\smallskip

Let $X$ be a spectral space (i.e., a topological space that is homeomorphic to the prime spectrum of a ring, endowed with the Zariski topology).   It is possibile to consider on  $X$ another topology (see \cite[Proposition 8]{ho}) defined by taking the collection of all the open and quasi-compact subspaces of $X$ as a basis of closed sets. This topology is called \textit{the inverse topology on $X$}. Note that, by definition, the closure of a subset $Y$ of $X$, with respect to the inverse topology, is given by
	$$
	\bigcap\{U\mid U\subseteq X\mbox{ open and quasi-compact, }Y\subseteq U \}\,.
	$$


\section{Stable semistar operations and localizing systems}\label{sect:stable}

A semistar operation $\star$ defined on an integral domain $D$ is called \emph{stable} provided that, for any $E,H\in \FF(D)$, we have $(E\cap H)^\star=E^\star\cap H^\star$. We denote by $\SStarstab(D)$ the set of stable semistar operations on $D$.

\begin{prop}\label{immersione:stab}
Let $D$ be an integral domain and let $\iota:\overr(D)\longrightarrow\SStarf(D)$ be the topological embedding defined in Proposition \ref{immersione}. If $T\in\overr(D)$, then $\iota(T)\in\SStarstab(D)$ if and only if $T$ is flat over $D$.
\end{prop}
\begin{proof}
It is enough to note that the equality $(I\cap J)T=IT\cap JT$ holds for every ideal $I,J$ of $D$ if and only if $T$ is flat (\cite[Proposition 1.7]{uda} and \cite[Theorem 7.4(i)]{matsu-86}).
\end{proof}

Given a semistar operation $\star$ on $D$, we can always associate to $\star$ a stable semistar operation $\stu$ by defining, for every $E\in\FF(D)$,
\begin{equation*}
E^{\stu}:=\bigcup \{(E:I)\mid I \mbox{ nonzero ideal of } D \mbox{ such that } I^\star =D^\star\}.
\end{equation*}
It is easy to see that $\stu \boldsymbol{\preceq} \star$ and, moreover, that $\stu$ is the largest stable semistar operation that precedes $\star$. Therefore, $\star$ is stable if and only if $\star = \stu$ \ \cite[Proposition 3.7, Corollary 3.9]{fohu}.

\begin{prop}\label{stable:subret}
Let $D$ be an integral domain, and denote by  $\insid$ the set of nonzero ideals of $D$. Let also $\overline{\Phi}:\SStar(D)\rightarrow\SStarstab(D)$ be the map defined by  $\overline{\Phi}(\star):=\stu$, for each $\star \in \SStar(D)$.
\begin{enumerate}[(1)]
\item\label{subbasis:stable} The set $\{\mbox{\rm\texttt{V}}_I\cap\SStarstab(D)\mid I\in\insid\}$ is a subbasis for $\SStarstab(D)$.
\item\label{retraction:stable} If $\SStar(D)$ is endowed with the Zariski topology, $\overline{\Phi}$ is a topological retraction.
\item\label{weak topology:stable} If $\SStar(D)$ is endowed with the topology generated by the family $\{\mbox{\rm\texttt{V}}_I\mid I\in\insid\}$, then $\overline{\Phi}$ is the canonical  map onto the Kolmogoroff quotient space of $\SStar(D)$.
\end{enumerate}
\end{prop}
\begin{proof}
\ref{subbasis:stable} For any nonzero $D$-submodule $E$ of $K$, and any stable semistar operation $\star$, we have $1\in E^\star$ if and only if $1\in E^\star\cap D^\star=(E\cap D)^\star$. Therefore, $\star \in \mbox{\rm\texttt{V}}_E$ if and only if $\star\in \mbox{\rm\texttt{V}}_{E\cap D}$. The claim follows.

\ref{retraction:stable} We claim that, if $I$ is an ideal of $D$, $\overline{\Phi}^{-1}(\mbox{\rm\texttt{V}}_I\cap\SStarstab(D))=\mbox{\rm\texttt{V}}_I$. Indeed, if $\star\in \mbox{\rm\texttt{V}}_I$ then $1\in I^\star$, so $1\in (I:I)\subseteq I^{\stu}$ and $\overline{\Phi}(\star)\in \mbox{\rm\texttt{V}}_I\cap\SStarstab(D)$. Conversely, if $\star\in\overline{\Phi}^{-1}(\mbox{\rm\texttt{V}}_I\cap\SStarstab(D))$, then $\overline{\Phi}(\star)\in \mbox{\rm\texttt{V}}_I$, and thus $1\in(I:E)$ for some $E\in\insid$ such that $E^\star=D^\star$. But this means that $E\subseteq I$, and so $1\in I^\star$, i.e., $\star\in \mbox{\rm\texttt{V}}_I$. Hence, $\overline{\Phi}$ is continuous.

Since $\star$ is stable if and only if $\star=\stu$ \cite[Proposition 3.7, Corollary 3.9]{fohu}, it follows that $\overline{\Phi}$ is a topological retraction.

\ref{weak topology:stable} Since, by the previous point, $\overline{\Phi}^{-1}(\mbox{\rm\texttt{V}}_I\cap\SStarstab(D))=
\mbox{\rm\texttt{V}}_I$,   the map $\overline{\Phi}$ is continuous even when $\SStar(D)$ is endowed with the weaker topology. To show that $\overline{\Phi}$  is the canonical  map onto the Kolmogoroff quotient space of $\SStar(D)$, it is enough to show that $\overline{\star_1}=\overline{\star_2}$ if and only if $\star_1\in \mbox{\rm\texttt{V}}_I$ is equivalent to $\star_2\in \mbox{\rm\texttt{V}}_I$.

Suppose $\overline{\star_1}=\overline{\star_2}$, and let $I$ be an ideal of $D$ such that $\star_1\in \mbox{\rm\texttt{V}}_I$, that is, $1\in I^{\star_1}$. By definition, $1\in(I:I)\subseteq I^{\overline{\star_1}}=I^{\overline{\star_2}}\subseteq I^{\star_2}$. It follows that $\star_2\in \mbox{\rm\texttt{V}}_I$. By symmetry, we deduce that if $\star_2\in \mbox{\rm\texttt{V}}_I$ then $\star_1\in \mbox{\rm\texttt{V}}_I$.

Conversely, suppose that $\star_1\in \mbox{\rm\texttt{V}}_I$ if and only if $\star_2\in \mbox{\rm\texttt{V}}_I$. Then, $I^{\star_1}=D^{\star_1}$ if and only if $I^{\star_2}=D^{\star_2}$; a direct application of the definition of the stable semistar operations canonically associated shows that $\overline{\star_1}=\overline{\star_2}$ (cf. \cite[page 182]{fohu}).
\end{proof}

\begin{oss} \label{eq:defstare}
 Recall that for any \emph{star} operation $\ast$ on an integral domain $D$ with quotient field $K$ and for any $E\in\FF(D)$, we can consider the map $\ast_e$ defined by
\begin{equation*}
E^{\ast_e}:=\begin{cases}
E^\ast & \text{if~}E\in \F(D),\\
K & \text{if~}E\in \FF(D)\setminus \F(D).
\end{cases}
\end{equation*}
The map $E\mapsto E^{\ast_e}$ defines a semistar operation on
$D$ such that $D^{\ast_e}=D$, called {\it the trivial semistar extension of} $\ast$. Note that, even if $\ast$ is a stable star operation, $\ast_e$ is not always stable: for example, let $D$ be a Dedekind domain with exactly two maximal ideals, $P$ and $Q$, and let $\ast$ be the identity star operation. Then, $D_P$ and $D_Q$ are not fractional ideals of $D$ \cite[Example 5.7]{FiSp}, and thus $D_P^{\ast_e}=K=D_Q^{\ast_e}$. On the other hand, $(D_P\cap D_Q)^{\ast_e}=D^{\ast_e}=D$.
\end{oss}

\begin{prop}\label{stable:stare}
Let $\ast$ be a stable \emph{star} operation on an integral domain $D$. There is exactly one stable semistar operation $\star$ on $D$ such that $\star|_{\F(D)}=\ast$.
\end{prop}
\begin{proof}
Suppose there exist two stable semistar extensions $\star_1$ and $\star_2$ of the star operation $\ast$, i.e., $ \star_1|_{\F(D)}= \star_2|_{\F(D)} =\ast$. Since $\SStarstab(D)$ is $T_0$, there is a subbasic open set $\mbox{\rm\texttt{U}}_I := \mbox{\rm\texttt{V}}_I \cap \SStarstab(D) $, with $I$ a proper ideal of $D$, such that $\star_1\in \mbox{\rm\texttt{U}}_I$ but $\star_2\notin \mbox{\rm\texttt{U}}_I$ (or conversely). But this would imply $I^\ast=I^{\star_1}\neq I^{\star_2}=I^\ast$, which is absurd.

For the existence, consider the semistar operation $\star:=\overline{\ast_e}$, where $\ast_e$ is the trivial semistar extension of $\ast$ defined in Remark \ref{eq:defstare}; by definition, $\star$ is a stable semistar operation. On the other hand, since $D^\star=D$, if $I$ is a nonzero $D$-submodule of $K$ such that $I^\star=D^\star$, then $I$ is an ideal in $D$. It follows that $\star|_{\F(D)}$ is the stable closure of $\ast$ as a \emph{star} operation, as defined in \cite[Definition 2.2]{ac}. However, since $\ast$ is already stable, we have $\star|_{\F(D)}=\ast$, i.e., $\star$ is an extension of $\ast$.
\end{proof}

Our next goal is to estabilish a topological connection between stable operations and localizing systems.

A \emph{localizing system} on $D$ is a subset $\mathcal{F}$ of ideals of $D$ such that:
\begin{itemize}
\item if $I\in\mathcal{F}$ and $J$ is an ideal of $D$ such that $I\subseteq J$, then $J\in\mathcal{F}$;
\item if $I\in\mathcal{F}$ and $J$ is an ideal of $D$ such that, for each $i\in I$, $(J:_D iD)\in\mathcal{F}$, then $J\in\mathcal{F}$.
\end{itemize}

A localizing system $\mathcal{F}$ is \emph{of finite type} if for each $I\in \mathcal{F}$ there exists a nonzero finitely generated ideal $J\in \mathcal{F}$ with $J\subseteq I$. For instance, if $T$ is an overring of $R$, $\mathcal{F}(T):=\{I \mid I$ ideal of $D,IT=T\}$ is a localizing system of finite type. On the other hand, if $V$ is a valuation domain and $P$ is a nonzero idempotent prime ideal of $V$, then $\hat{\mathcal{F}}(P): =\{I \mid I\mbox{ ideal of } V\mbox{ and }I\supseteq P\}$ is a localizing system of $V$ which is not of finite type. Given a localizing system $\mathcal{F}$ of an integral domain $D$, then
\begin{equation*}
{\mathcal{F}}_{\!{_f}}:=\{ I\in \mathcal{F} \mid I\supseteq J, \mbox{ for some nonzero finitely generated ideal } J\in \mathcal{F}\}
\end{equation*}
is a localizing system of finite type of $D$, and ${\mathcal{F}} ={\mathcal{F}}_{\!{_f}}$ if and only if $\mathcal{F}$ is a localizing system of finite type.

We denote by $\locsist(D)$ [respectively, $\locsistf(D)$] the set of all localizing systems [respectively, localizing systems of finite type] on $D$. For further details on localizing systems, see \cite[Chap. II, \S 2, Exercices 17--25]{BAC} or \cite[Sections 2 and 3]{fohu}. 

It is well known that, to each localizing system $\mathcal{F}$, we can associate a semistar operation $\star_{\mathcal{F}}$ defined as follows, for each $E\in\FF(D)$,
\begin{equation*}
E^{\star_{\mathcal{F}}}:=\bigcup\{(E:H)\mid H \in\mathcal{F}\}.
\end{equation*}
The assignment $\mathcal{F}\mapsto\star_{\mathcal{F}}$ defines a map $\boldsymbol{\lambda}:\locsist(D)\rightarrow\SStar(D)$. By \cite[Theorem 2.10, Corollary 2.11, and Proposition 3.2]{fohu}, $\boldsymbol{\lambda}$ is an injective map whose image is exactly $\SStarstab(D)$.

On the set $\locsist(D)$ of localizing sytem on $D$ we can introduce a natural topology, that we still call the \emph{Zariski topology}, whose subbasic open sets are the $\mbox{\rm\texttt{W}}_I:=\{\mathcal{F}\in \locsist(D)\mid I\in\mathcal{F}\}$, as $I$ varies among the ideals of $D$.

\begin{prop} \label{locsist-stable}
Let $D$ be an integral domain. The map $\boldsymbol{\lambda}:\locsist(D)\rightarrow \SStarstab(D)$, $\mathcal{F}\mapsto\star_{\mathcal{F}}$, establishes a homeomorphism between spaces endowed with the Zariski topologies.
\end{prop}
\begin{proof}
By the previous remarks, we only need to show that $\boldsymbol{\lambda}$ is continuous and open.
Let $\mbox{\rm\texttt{U}}_I$ be a subbasic open set of $\SStarstab(D)$. Then,
$$ 
\begin{array}{rl}
\boldsymbol{\lambda}^{-1}(\mbox{\rm\texttt{U}}_I)=&\{\mathcal{F}\in \locsist(D) \mid\star_{\mathcal{F}}\in \mbox{\rm\texttt{U}}_I\}= \\
=& \{\mathcal{F}\in \locsist(D)\mid 1\in(I:H)\text{~for some~}H\in\mathcal{F}\}= \\
=& \{\mathcal{F}\in \locsist(D)\mid H\subseteq I\text{~for some~}H\in\mathcal{F}\}=\bigcup_{H\subseteq I}\mbox{\rm\texttt{W}}_H ,
\end{array}
$$
and thus $\boldsymbol{\lambda}$ is open. Moreover, $\bigcup_{H\subseteq I}\mbox{\rm\texttt{W}}_H=\mbox{\rm\texttt{W}}_I$: indeed, if $\mathcal{F}\in \mbox{\rm\texttt{W}}_H$ then $H\in\mathcal{F}$, and so $I\in\mathcal{F}$, while the left hand union trivially contains $\mbox{\rm\texttt{W}}_I$. Therefore, $\boldsymbol{\lambda}^{-1}(\mbox{\rm\texttt{U}}_I)=\mbox{\rm\texttt{W}}_I$, and, since $\boldsymbol{\lambda}$ is bijective, $\boldsymbol{\lambda}(\mbox{\rm\texttt{W}}_I)=\mbox{\rm\texttt{U}}_I$. Therefore, $\boldsymbol{\lambda}$ is both continuous and open, and thus a homeomorphism.
\end{proof}


\section{Spectral semistar operations}\label{sect:spectral}
 If $Y$ is a 
subset of the prime spectrum $\spec(D)$ of an integral domain $D$, then we define the semistar operation $\mbox{\it\texttt{s}}_Y$ \emph{induced} by $Y$ as the semistar operation associated to the set $\boldsymbol{\mathcal{T}}(Y) := \{D_P \mid P \in Y\}$, i.e.,
\begin{equation*}
E^{\mbox{\it\tiny\texttt{s}$_Y$}}:=\bigcap \{ED_{P} \mid P\in Y\},\;  \mbox{ for every } E\in \FF(D).
\end{equation*}
 If $Y=\emptyset$, we  have  an empty intersection, and we set as usual 
 $E^{s_\emptyset}:=K$ for every $E\in\FF(D)$ (or, equivalently,  $\mbox{\it\texttt{s}}_\emptyset := \wedge_{\{K\}}$).

 A semistar operation of the type $\mbox{\it\texttt{s}}_Y$, for some $Y \subseteq \spec(D)$, is called a \emph{spectral} semistar operation.  

 Note that, if we take $Y = \{(0)\}$,  we also have that $\mbox{\it\texttt{s}}_Y = \wedge_{\{K\}}$. Therefore, without loss of generality, in the definition of spectral semistar operation we can assume $ \emptyset \neq Y\subseteq \spec(D)$.

Denote by $\SStarsp(D)$ the set of spectral semistar operations, and by  $\SStarstabft(D)$ the set of spectral semistar operations of finite type.

Since each localization of $D$ is $D$-flat, and the infimum of a family of stable semistar operation is again stable, using Proposition \ref{immersione:stab} we see that every spectral semistar operation is stable. On the other hand, not every stable operation is spectral; 
however, a stable semistar operation is spectral if and only if it is  quasi-spectral \cite[Proposition 4.23(2)]{fohu}. 
In particular, every finite type stable operation is spectral (cf. \cite[Corollary 4.2]{an-overrings} and \cite[page 185 and Theorem 4.12(3)]{fohu}), so that $\SStarstabft(D)$  coincides with
 the set of stable operations of finite type  (sometimes denoted by ${\widetilde{\mbox{\rm\texttt{SStar}}}}(D)$).

Like for $\stu$, we can associate to each semistar operation $\star$ a stable semistar operation of finite type $\stt$ by defining, for every $E\in\FF(D)$,
\begin{equation*}
\begin{array}{rl}
E^{\stt}:= & \bigcup \{ (E:J)\mid J \mbox{ nonzero finitely generated ideal of } D \\
& \hskip 50pt \mbox{ such that }J^\star =D^\star\}.
\end{array}
\end{equation*}
The stable semistar operation of finite type $\stt$ is smaller than $\star$, and it is the biggest stable semistar operation of finite type smaller than $\star$. It follows that $\star$ is stable of finite type if and only if $\star=\stt$. 

\smallskip

In the following proposition, we collect some of the properties concerning the relation between $Y$ and $\mbox{\it\texttt{s}}_Y$.
As usual, for each subset $Y$ of  $\spec(D)$, we set $Y^{\mbox{\rm\tiny\mathttt{gen}}}:= \{z \in \spec(D) \mid y \in \chius(\{z\}) \mbox{ for some } y \in Y \}$ and we denote $\mbox{\rm{\texttt{Cl}}}^{\mbox{\rm\tiny\mathttt{inv}}}(Y)$ the closure of $Y$ in the inverse topology of  $\spec(D)$.

\begin{prop}\label{spettrali} {\em{(cf. \cite[Corollaries 4.4 and 5.2, Proposition 5.1]{FiSp} and \cite[Lemma 4.2 and Remark 4.5]{fohu}})}
	Let $D$ be an integral domain and let $Y$ and $Z$ be two nonempty subsets of $\spec(D)$. The following statements hold.  
\begin{enumerate}[\rm(1)]
\item $\mbox{\it\texttt{s}}_Y = \mbox{\it\texttt{s}}_Z$ if and only if  $Y^{\mbox{\rm\tiny\texttt{gen}}} = Z^{\mbox{\rm\tiny\texttt{gen}}}$.
\item $\mbox{\it\texttt{s}}_Y$ is of finite type if and only if $Y$ is quasi-compact. 
\item  $\widetilde{\mbox{\it\texttt{s}}_Y}=\widetilde{\mbox{\it\texttt{s}}_Z}$ if and only if    $\Cl^{\mbox{\tiny\rm\texttt{inv}}}(Y) =\Cl^{\mbox{\tiny\rm\texttt{inv}}}(Z)$.		
\item[\rm (4)] $\widetilde{\mbox{\it\texttt{s}}_Y}=\mbox{\it\texttt{s}}_{{\tiny \Cl}^{\mbox{\tiny\rm\texttt{inv}}}(Y)}$.
\end{enumerate}
\end{prop}

\smallskip

\begin{oss} \label{gen} Note that the equivalence (1) of Proposition  \ref{spettrali} can   also be viewed in a topological way, since the sets of the type 
 $Y^{\mbox{\rm\tiny\mathttt{gen}}}$, for $Y$ varying among the subsets of $\spec(D)$, are the closed sets of a topology called the \emph{R(ight)-topology on} $\spec(D)$, which is the finest topology on $\spec(D)$ compatible with the opposite order of the given order on $\spec(D)$ \cite[Lemma 2.1, Proposition 2.3(b)]{dofopa-80}.
 \end{oss}
 
\smallskip

The following is a ``finite type version'' of Propositions \ref{stable:subret}, \ref{stable:stare} and \ref{locsist-stable}.

\begin{prop}\label{stableft}
Let $D$ be an integral domain, and denote by\ $\insid_{\boldsymbol{f}}$ the set of nonzero finitely generated ideals of $D$. Let also $\widetilde{\Phi}:\SStar(D)\rightarrow\SStarstabft(D)$ be the map defined by  $\widetilde{\Phi}(\star):=\stt$, for each $\star \in \SStar(D)$.
\begin{enumerate}[\rm(1)]
\item\label{subbasis:stableft} The set $
\{ \widetilde{\mbox{\rm\texttt{{U}}}}_J:=\mbox{\rm\texttt{V}}_J\cap\SStarstabft(D)\mid J\in\insid_{\boldsymbol{f}}\}$
 is a subbasis   of open and quasi-compact subspaces for $\SStarstabft(D)$.
\item\label{retraction:stableft} If $\SStar(D)$ is endowed with the Zariski topology, $\widetilde{\Phi}$ is a topological retraction.
\item\label{weak topology:stableft} If $\SStar(D)$ is endowed with the topology generated by the family $\{\mbox{\rm\texttt{V}}_J\mid J\in\insid_{\boldsymbol{f}}\}$, then $\widetilde{\Phi}$ is the canonical  map onto the Kolmogoroff quotient space of $\SStar(D)$.
 
\item\label{stare:stableft} If $\ast$ is a stable \emph{star} operation of finite type on $D$, there is exactly one stable semistar operation of finite type $\star$ on $D$ such that $\star|_{\F(D)}=\ast$.
\item\label{locsist-stableft} The restriction $\boldsymbol{\lambda_f}:\locsistf(D)\rightarrow \SStarstabft(D)$ of $\boldsymbol{\lambda}$ establishes a homeomorphism between spaces endowed with the Zariski topologies.
\end{enumerate}
\end{prop}
\begin{proof}
The proofs follow essentially from the general case, with some additional care. For \ref{subbasis:stableft}, we note that, for each $I \in \insid$,  $\mbox{\rm\texttt{V}}_I\cap\, \SStarstabft(D)=\bigcup\{\mbox{\rm\texttt{V}}_J\cap\, \SStarstabft(D)\mid J\subseteq I, \ J\in\f(D)\}$ (compare \cite[Remark 2.2(d)]{FiSp}).   To show that $\widetilde{\mbox{\rm\texttt{{U}}}}_J$ is quasi-compact when  $J \in \insid_{\boldsymbol{f}}$, consider the semistar operation $s_Y$, where $Y:=\{P\in \spec(D)\mid P\nsupseteq J \}$.  Since $Y$ is quasi-compact, $s_Y$ is of finite type (see \cite[Corollary 4.4]{FiSp} or Proposition \ref{spettrali}(2)). Moreover, $1\in JD_P$ if and only if $P\in Y$; therefore, $s_Y$ is the minimum of $\widetilde{\mbox{\rm\texttt{{U}}}}_J$, and every open set containing $s_Y$ contains the whole $\widetilde{\mbox{\rm\texttt{{U}}}}_J$. It follows that $\widetilde{\mbox{\rm\texttt{{U}}}}_J$ is quasi-compact.

For \ref{stare:stableft}, it is enough to note that if $\ast$ is of finite type then $\overline{\ast_e}=\widetilde{\,\ast_e}$, so that $\star$ is of finite type if $\ast$ is.
For \ref{locsist-stableft}, note that the image of $\boldsymbol{\lambda_f}$ is exactly $\SStarstabft(D)$   \cite[Theorem 2.10, Corollary 2.11, and Proposition 3.2]{fohu}.
\end{proof}

The remaining part of the present section is devoted to the proof that $\SStarstabft(D)$ is a spectral space, in the sense of M. Hochster \cite{ho}. We start by studying the supremum and the infimum of a family of spectral operations.

\begin{lemma}\label{prop:infsup}
Let $\mathscr{D}$ be a nonempty set of spectral semistar operations. For each spectral semistar operation $\star $, set $\Delta(\star) := \qspec^\star(D)$.

\begin{enumerate}[\rm(1)]
\item\label{prop:infsup:inf} $\wedge_{\mathscr{D}}$ 
is spectral with $\Delta(\wedge_{\mathscr{D}})=
 \bigcup \{\Delta(\star) \mid \star \in \mathscr{D}\}$.
\item\label{prop:infsup:sup} If $\vee_\mathscr{D}$ is quasi-spectral, then is spectral with $\Delta(\vee_{\mathscr{D}})=
 \bigcap \{\Delta(\star) \mid \star \in \mathscr{D}\}$.
\end{enumerate}
\end{lemma}
\begin{proof}
\ref{prop:infsup:inf} Set $\boldsymbol{\Delta} := \bigcup \{\Delta(\star) \mid \star \in \mathscr{D}\}$. For each $E \in \FF(D)$ 
\begin{equation*}
E^{\wedge_{\mathscr{D}}} =
\bigcap\{ ED_P \mid P\in \Delta(\star),\star \in \mathscr{D}\} =
 \bigcap\{ ED_P \mid P\in \boldsymbol{\Delta}\} . 
 \end{equation*}
 In particular, ${\wedge_{\mathscr{D}}}$ is spectral and $\boldsymbol{\Delta}\subseteq\qspec^{\wedge_{\mathscr{D}}}(D)$. On the other hand, if $Q\in\qspec^{\wedge_{\mathscr{D}}}(D)$, then $Q^\star\neq D^\star$ for some $\star\in \mathscr{D}$, and this implies that $Q\in\qspec^\star(D)$. 
 Therefore, $\boldsymbol{\Delta}=\qspec^{\wedge_{\mathscr{D}}}(D)$.

\ref{prop:infsup:sup} Let $P\in\qspec^{\vee_\mathscr{D}}(D)$. Then, $P$ belongs to $\qspec^{\star}(D)$ 
 for each $\star\in \mathscr{D}$, i.e., $\qspec^{\vee_\mathscr{D}}(D)\subseteq
 \bigcap \{\Delta(\star) \mid \star \in \mathscr{D}\}$.  Since each $\star \in \mathscr{D}$ is spectral, then, for each $E\in\FF(D)$, $E^\star\subseteq ED_P$ for all $P\in\Delta(\star)$, and in particular for all $P\in\bigcap\{\Delta(\star) \mid \star\in\mathscr{D}\}$. Hence,
 \begin{equation*}
E^{\vee_\mathscr{D}}\subseteq \bigcap \left\{ED_P \mid P\in\bigcap\{\Delta(\star) \mid \star\in\mathscr{D}\}\right\} \subseteq \bigcap \{ED_P \mid P \in\qspec^{\vee_\mathscr{D}}(D)\}.
\end{equation*}
However, if ${\vee_\mathscr{D}}$ is quasi-spectral, 
then it is known that the right hand side is contained in $E^{\vee_\mathscr{D}}$ \cite[Proposition 4.8]{fohu}.
 Therefore they are equal, and hence ${\vee_\mathscr{D}}$ is spectral, with $ \Delta(\vee_{\mathscr{D}})=
 \bigcap \{\Delta(\star) \mid \star \in \mathscr{D}\}$.
\end{proof}

\begin{ex}
In relation with Lemma \ref{prop:infsup}(2), we note that the supremum of a family of spectral semistar operations may not be quasi-spectral. Indeed, let $\insA$ be the ring of algebraic integers, i.e., the integral closure of $\insZ$ in the algebraic closure $\overline{\insQ}$ of $\insQ$. Recall that $\insA$ is a one-dimensional B\'ezout domain \cite[page 72]{kap}.

{\bf Claim 1.} \emph{For each maximal ideal $P$ of $\insA$, $\Max(\insA)\setminus\{P\}$ is not a quasi-compact subspace of $\Max(\insA)$ (endowed with the Zariski topology).}

By contradiction, since $\Max(\insA)\setminus\{P\}$ is open in $\Max(\insA)$ it would be equal to $\texttt{D}(J) \cap \Max(\insA)$ for some finitely generated ideal $J$ of $\insA$. Being $\insA$ a B\'ezout domain, this would imply that $\Max(\insA)\setminus\{P\}= \texttt{D}(\alpha) \cap \Max(\insA)$ for some $\alpha\in \insA$; in particular, the ideal $\alpha A$ would be $P$-primary. Let $K$ be the Galois closure of $\insQ(\alpha)$ over $\insQ$ and consider the prime ideal $P_K:=P\cap\mathcal{O}_K$, where $\mathcal{O}_K$ the ring of integers of the field $K$. Let $P\cap\insZ=p\insZ$ for some prime integer $p$ and let $F$ be a Galois extension of $\insQ$ where $p$ splits and such that $F\cap K=\insQ$ (there are infinitely many such fields $F$, since $p$ splits in infinitely many quadratic extensions of $\insQ$ and $K$ contains only a finite number of them). We claim that $P_K$ splits in the compositum $F\!K$: if this is true, then $\alpha$ would be contained in more than a single prime ideal of $\insA$, against the hypothesis.

 Set $P_{F\! K}:=P\cap\mathcal{O}_{F\! K}$ and $P_F:= P\cap\mathcal{O}_{F}=P_{F\! K}\cap\mathcal{O}_F $. Suppose $P_K$ does not split in $\mathcal{O}_{F\! K}$: then $P_K\mathcal{O}_{F\! K}$ would be primary to $P_{F\! K}$. 
 On the other hand, $P_F\cap\insZ=p\insZ$; since $p$ splits in $\mathcal{O}_F$, and the Galois group of $F$ over $\insQ$ acts transitively on the primes of $\mathcal{O}_F$ lying over $p$, there is an automorphism $\sigma$ of $F$ such that $\sigma(P_F)\neq P_F$. Since $K\cap F=\insQ$, there is an automorphism $\tau$ of $FK$ such that $\tau|_F=\sigma$ and $\tau|_K$ is the identity. Therefore, $\tau(P_K)=P_K$ and $\tau(P_{F\! K})$ must contain $P_K$, i.e., $\tau(P_{F\! K})= P_{F\! K}$. 
 However, $P_{F\! K}$ contains $P_F$, and $\tau(P_{F\! K})$ contains $\sigma(P_F)$; therefore, $P_{F\! K}$ must contain both $P_F$ and $\sigma(P_F)$, which is impossible. Therefore, $P_K$ splits in $\mathcal{O}_{F\! K}$.

{\bf Claim 2.} \emph{For every $P\in\Max(\insA)$, $\bigcap \{ \insA_Q \mid Q\in\Max(\insA)\setminus\{P\}\} =\insA$.}

Let $B:=\bigcap \{ \insA_Q \mid Q\in\Max(\insA)\setminus\{P\}\}$. By the previous claim, $\Max(\insA)\setminus\{P\}$ is not quasi-compact,  and then it follows immediately that $P$ belongs to the closure of $\Max(\insA)\setminus\{P\}$, with respect to the inverse topology.    In other words,  every maximal ideal of $\insA$ is a limit point in the inverse topology and so  $\Max(\insA)$ with the inverse topology is a perfect space. Finally,  by \cite[Proposition 5.6(4)]{olb2015}, we have $B=\insA$.  

We are ready now to show that the supremum of a family of spectral semistar operations on $\insA$ may not be quasi-spectral. For every $P\in\Max(\insA)$, let $\star_P:=\mbox{\it\texttt{s}}_{\Max(\insA)\setminus\{P\}}$, and define $\star:=\bigvee \{\star_P \mid P\in\Max(\insA)\}$. By Claim 2, $\insA^{\star_P}= \insA$ for every $P\in\Max(\insA)$, and thus $\insA^{\star}= \insA$. However, $P$ is not a $\star_P$-ideal since $P^{\star_P} = \insA$ and therefore $P^{\star} = \insA$ for every $P\in\Max(\insA)$. Since each nonzero principal (or, equivalently, finitely ge\-ne\-ra\-ted) integral ideal of $\insA$ is a $\star$-ideal and the set of nonzero prime $\star$-ideals of $\insA$ is empty, it follows that 
$\star$ is not quasi-spectral.
\end{ex}

\begin{teor}\label{stable spectral}
Let $D$ be an integral domain. The space $ \SStarstabft(D)$ of the stable semistar operations of finite type on $D$, endowed with the Zariski topology induced by $\SStar(D)$, is a spectral space. 
\end{teor}
\begin{proof} In order to prove that a topological space $X$ is a spectral space, we use the characterization given in \cite[Corollary 3.3]{Fi}. We recall that if $\mathscr{B}$ is a nonempty family of subsets of $X$, for a given subset $Y$ of $X$
and an ultrafilter $\ms{U}$ on $Y$, we set
$$
Y_\mathscr{B}(\ms{U}) := \{x \in X \mid\mbox{ for each } \ B \in \mathscr{B}, \mbox{ it happens that } \; x \in B \Leftrightarrow B \cap Y \in \ms{U}\} .
$$ 
The subset $Y$ of $X$ is called {\it $\mathscr{B}$-ultrafilter closed} if $Y_\mathscr{B}(\ms{U}) \subseteq Y$; the $\mathscr{B}$-ultrafilter closed subsets of $X$ are the closed subspaces of a topology on $X$ called the {\it $\mathscr{B}$-ultrafilter topology on} $X$.

By \cite[Corollary 3.3]{Fi},
for a topological space $X $ being a spectral space is equi\-va\-lent to $X$   being a $T_0$-space having a subbasis for the open sets $\mathscr{S}$ such that
$X_{\mathscr{S}}(\ms{U}) \neq \emptyset$, for each ultrafilter $\ms{U} $ on $X$.

We already know that $\boldsymbol{\mathcal{X}}:= \SStarstabft(D)$ is a $T_0$-space. 
By  Proposition \ref{stableft}\ref{subbasis:stableft}, the collection of sets
\begin{equation*}
\boldsymbol{\mathcal{\widetilde{T}}}:=\{\widetilde{\mbox{\rm\texttt{{U}}}}_J \mid J\subseteq D,J \mbox{ finitely generated ideal of }D \}
\end{equation*}
is a subbasis of the Zariski topology of $\boldsymbol{\mathcal{X}}$.
 Let $\ms U$ be any ultrafilter on $\boldsymbol{\mathcal{X}}$; the conclusion will follow if we prove that the set
\begin{equation*}
\boldsymbol{\mathcal{X}}_{\boldsymbol{\mathcal{\widetilde{T}}}}(\ms U):=\{\star\in \boldsymbol{\mathcal{X}} \mid
   [\mbox{ for each } 
   \widetilde{\mbox{\rm\texttt{{U}}}}_J
    \in \boldsymbol{\mathcal{\widetilde{T}}},
     \mbox{ it happens that } \star\in \widetilde{\mbox{\rm\texttt{{U}}}}_J
   \Leftrightarrow 
   \widetilde{\mbox{\rm\texttt{{U}}}}_J \in\ms U]   \}
\end{equation*}
is nonempty. 

Consider the semistar operation 
\begin{equation*}
\bigstar:=\bigvee \{ \wedge_{\widetilde{\mbox{\rm\Small\texttt{{U}}}}_J} \mid \widetilde{\mbox{\rm\texttt{{U}}}}_J \in\ms U\},
\end{equation*}
on $D$. By Lemma \ref{prop:infsup}(1),
each $\wedge_{\widetilde{\mbox{\rm\Small\texttt{{U}}}}_J}$ is spectral and, since $\widetilde{\mbox{\rm\texttt{{U}}}}_J$ is quasi-compact   (Proposition \ref{stableft}(1)), is also of finite type \cite[Proposition 2.7]{FiSp}.
 By applying \cite[p. 1628]{an}, it can be easily shown that $\bigstar$ is of finite type, and thus it is quasi-spectral \cite[Corollary 4.21]{fohu}.
 By Lemma \ref{prop:infsup}(2), it follows that $\bigstar$ is spectral, i.e., $\bigstar \in \boldsymbol{\mathcal{X}}$.

To prove that $\bigstar\in \boldsymbol{\mathcal{X}}_{\boldsymbol{\mathcal{\widetilde{T}}}}(\ms U)$, we apply the same argument used in proving \cite[Theorem 2.13]{FiSp}: let $\widetilde{\mbox{\rm\texttt{{U}}}}_J\in\ms{U}$. If $\bigstar\in\widetilde{\mbox{\rm\texttt{{U}}}}_J$, then by the definition of $\bigstar$ and \cite[p.1628]{an} (see also \cite[Lemma 2.12]{FiSp}) there are finitely generated ideals $F_1, F_2,\ldots,F_n$ of $D$ such that 
$$1\in J^{\wedge_{\widetilde{\mbox{\rm\Small\texttt{{U}}}}_{F_1}}\circ \wedge_{\widetilde{\mbox{\rm\Small\texttt{{U}}}}_{F_2}}\circ\cdots\circ\wedge_{\widetilde{\mbox{\rm\Small\texttt{{U}}}}_{F_n}}},
$$ 
with  $\widetilde{\mbox{\rm\texttt{{U}}}}_{F_i}\in\ms{U}$. 
Each semistar operation $\sigma\in\widetilde{\mbox{\rm\texttt{{U}}}}_{F_1}\cap \widetilde{\mbox{\rm\texttt{{U}}}}_{F_2}\cap\cdots \cap\widetilde{\mbox{\rm\texttt{{U}}}}_{F_n}$ is bigger than each each semistar operation of type $\wedge_{\widetilde{\mbox{\rm\Small\texttt{{U}}}}_{F_i}}$; it follows that $\widetilde{\mbox{\rm \texttt{{U}}}}_{F_1}\cap \widetilde{\mbox{\rm \texttt{{U}}}}_{F_2}\cap\cdots \cap \widetilde{\mbox{\rm \texttt{{U}}}}_{F_n}\subseteq\widetilde{\mbox{\rm\texttt{{U}}}}_J$, and thus the latter set is in $\ms{U}$. Conversely, if $\widetilde{\mbox{\rm \texttt{{U}}}}_J\in\ms{U}$, then $\bigstar\geq \wedge_{\widetilde{\mbox{\rm\Small\texttt{{U}}}}_J}$, and thus $1\in J^\bigstar$.  The proof is now complete. 
\end{proof}

 \begin{oss}\label{oss:spectral-constr}
The proof of Theorem \ref{stable spectral} actually shows more than just the fact that $ \SStarstabft(D)$ is a spectral space. Given a spectral space $X$, the \emph{constructible topology} on $X$ is the coarsest topology such that every open and quasi-compact subset of $X$ (in the original topology) is both open and closed.  By \cite{folo-2008}, the closed sets of the constructible topology in $X$ are the subsets $Y$ of $X$ such that, for every ultrafilter $\mathscr{U}$ of $Y$,
\begin{equation*}
Y_\mathscr{B}(\ms{U}) := \{x \in X \mid\mbox{ for each } B \in \mathscr{B}, \mbox{ it happens that } \; x \in B \Leftrightarrow B \cap Y \in \ms{U}\}\subseteq Y,
\end{equation*}
where $\mathscr{B}$ is the set of open and quasi-compact subspaces of $X$. Therefore, in view of the proof of \cite[Theorem 2.13]{FiSp},  what we have actually proved in Theorem \ref{stable spectral} is that, when $\SStarf(D)$ is endowed with the constructible topology, $\SStarstabft(D)$ is a closed subspace.
\end{oss}

From the previous theorem and Proposition \ref{stableft}\ref{locsist-stableft}, we deduce immediately the following: 

 \begin{cor} Let $D$ be an integral domain. The space $\locsistf(D)$ of the localizing systems of finite type on $D$, endowed with the Zariski topology induced by $\locsist(D)$, is a spectral space. 
 \end{cor}

\section{The space of {\rm \texttt{eab}} semistar operations of finite type}\label{sect:eab}

A semistar operation $\star$ on an integral domain $D$ is said to be an \emph{\texttt{eab} semistar operation} [respectively, an \emph{\texttt{ab} semistar operation}] if, for every $F,G,H\in\f(D)$ [respectively, for every $F\in\f(D)$, $G,H\in\FF(D)$] the inclusion $(FG)^\star\subseteq(FH)^\star$ implies $G^\star\subseteq H^\star$. Note that, if $\star$ is \texttt{eab}, then $\stf$ is also \texttt{eab}, since $\star$ and $\stf$ agree on finitely generated fractional ideals. The concepts of \texttt{eab} and \texttt{ab} operations coincide on finite type operations, but not in general \cite{fo-lo-2009,fo-lo-ma}.

\begin{oss}
W. Krull only considered the concept of ``\texttt{a}rithmetisch
\texttt{b}rauchbar'' operation (for short \texttt{ab}-operation, as above) \cite{Krull:1936}. He did not consider the concept of ``\texttt{e}ndlich \texttt{a}rithmetisch \texttt{b}rauchbar'' operation (or, more simply, \texttt{eab}-operation as above), that instead stems from the original version of Gilmer's book \cite{gi-1968}.
\end{oss}

Let $\Zar(D):= \{ V\mid V \mbox{ is a valuation overring of } D \}$ be equipped with the {\it Zariski topology}, i.e., the topology having, as subbasic open subspaces, the subsets $\Zar(D[x])$ for $x$ varying in $K$.  The set  $\Zar(D)$, endowed with the Zariski topology, is often called the  {\it Riemann-Zariski space of }$D$  \cite[Chapter VI, \S 17, page 110]{zs}.   Recently, the use of Riemann-Zariski spaces had a strong impact on the study of algebraic properties of integrally closed domains. For a deeper insight on this topic see, for example, \cite{ol,ol1,olb2015,olb_noeth}.

 A \emph{valuative semistar operation} is a semistar operation  of the type $\wedge_{Y}$, where $Y \subseteq \Zar(D)$; it is easy to see that it  is an \texttt{eab} semistar operation. In particular, the $\mbox{\it\texttt{b}}$-operation, where $\mbox{\it\texttt{b}} := 
\wedge_{\Zar(D)}$, is an \texttt{eab} semistar operation of finite type on $D$, since $\Zar(D)$ is quasi-compact \cite[Proposition 4.5]{FiSp}. More generally, for the same reason, for each overring $T$ of $D$, the valuative semistar operation $\mbox{\it\texttt{b}}(T) := \wedge_{\Zar(T)}$ is an \texttt{eab} semistar operation of finite type on $D$.

Just like in the case of the relation between stable and spectral operations, not every \texttt{eab} semistar operation is valutative, but the two definitions agree on finite type operations (see, for instance, \cite[Corollary 5.2]{folo-2001}). However, unlike the spectral case, there are example of quasi-spectral  \texttt{eab} operations that are not valutative \cite[Example 15]{fo-lo-2009}.

Denote by $\SStarval(D)$ [respectively, $\SStareab(D)$;   $\SStareabf(D)$] the set of valutative [respectively, \texttt{eab};  \texttt{eab} of finite type] semistar operations on $D$, endowed with the Zariski topology induced from $\SStar(D)$. By the previous remarks, we have:
$$
\SStareabf(D) :=\SStareab(D) \cap\SStarf(D) = \SStarval(D) \cap\SStarf(D)\,.
$$

To every semistar operation $\star\in\SStar(D)$ we can associate a map $\sta: \FF(D) \rightarrow \FF(D)$ defined by 
\begin{equation*}
F^{\sta}:=\bigcup\{((FG)^\star:G^\star)\mid G\in\f(D)\}
\end{equation*}
for every $F\in\f(D)$, and then extended to arbitrary modules $E\in \FF(D)$ by setting $E^{\sta} :=\bigcup\{F^{\sta}\mid F\subseteq E, \ F\in\f(D)\}$. 
The map $\sta$ is always an \texttt{eab} semistar operation of finite type on $D$ \cite[Proposition 4.5(1, 2)]{folo-2001}. Moreover, $\star= \sta$ if and only if $\star$ is an \texttt{eab} semistar operation of finite type, called {\it the \texttt{eab} semistar operation of finite type associated to $\star$} and, if $\star$ is an \texttt{eab} semistar operation, then $\sta =\stf$ \cite[Proposition 4.5(4)]{folo-2001}. The following proposition is an analogue of Propositions \ref{stable:subret}\ref{retraction:stable} and \ref{stableft}\ref{retraction:stableft}.

\begin{prop}\label{phia}
Let $D$ be an integral domain and let $\Phi_{\mbox{\it\tiny\mathttt{a}}}:\SStar(D)\rightarrow \SStareabf(D)$ be the map defined by $\Phi_{\mbox{\it\tiny\mathttt{a}}}(\star) :=\sta$, for each $\star\in \SStar(D)$. Then, $\Phi_{\mbox{\it\tiny\mathttt{a}}}$ is topological retraction of $\SStar(D)$ onto $\SStareabf(D)$.
\end{prop}
\begin{proof}
We start by showing that $\Phi_{\mbox{\it\tiny\mathttt{a}}}$ is continuous. Indeed, if $H$ is a nonzero finitely generated fractional ideal of $D$,
\begin{equation*}
\begin{array}{rl}
\Phi_{\mbox{\it\tiny\mathttt{a}}}^{-1}(\texttt{V}_H\cap \SStareabf(D))= & 
\hskip -6pt \{\star \in \SStar(D) \mid 1\in H^{\sta}\}\\
= & \hskip -6pt \{\star \in \SStar(D)\mid F^\star\subseteq(HF)^\star\mbox{ for some } F\in \f(D)\} \\
= & \hskip -6pt \bigcup \left\{\{\star \in \SStar(D) \mid F^\star\subseteq(HF)^\star\} \mid F\in \f(D) \right\}.
\end{array}
\end{equation*}

Let $F:=x_1D+x_2D+\cdots+x_nD$. Then, $F^\star\subseteq(HF)^\star$ is and only if $x_i\in(HF)^\star$ for $i=1,\ldots,n$, hence, 
\begin{equation*}
\Phi_{\mbox{\it\tiny\mathttt{a}}}^{-1}(\texttt{V}_H\cap\SStareabf(D)) =
\bigcup\left\{\bigcap_{i=1}^n \texttt{V}_{x_i^{-1}HF} \mid F \in \f(D) \right\},
\end{equation*}
which is an open set of $\SStar(D)$. Hence, $\Phi_{\mbox{\it\tiny\mathttt{a}}}$ is continuous.
Moreover, if $\star$ is an \texttt{eab} operation of finite type, then $\sta=(\star_{\mbox{\it\tiny\mathttt{a}}})_{\mbox{\it\tiny\mathttt{a}}}$.
Henceforth, $\Phi_{\mbox{\it\tiny\mathttt{a}}}$ is a topological retraction.
\end{proof}

We are now interested to what happens to the topological embedding $\iota$, defined in Proposition \ref{immersione}, when restricted to subsets of integrally closed overrings. 
We start with a remark.  

\begin{oss}\label{rem:stara}
Let $T$ be an overring of $D$, and let $\star_T$ be a semistar operation on $T$. Then, we can define a semistar operation $\star$ on $D$ by $\star:=\star_T\circ\wedge_{\{T\}}$, i.e., $E^\star:=(ET)^{\star_T}$ for every $E\in\FF(D)$. If now $F\in\f(T)$, then
$$
\begin{array}{rl}
F^{\sta} =& \hskip -6pt \bigcup \{((FG)^\star:G^\star) \mid G\in\f(D) \} =
   \bigcup \{((FGT)^{\star_T}:(GT)^{\star_T}) \mid G\in\f(D)\}=\\
=& \hskip -6pt \bigcup\{ ((FTH)^{\star_T}:H^{\star_T}) \mid H\in\f(T)\}=(FT)^{(\star_T)_{\mbox{\it\tiny\texttt{a}}}}=F^{(\star_T)_a}.
\end{array}
$$
Hence, for every $E\in\FF(D)$, $E^{\sta}=(ET)^{(\star_T)_a}$, that is, $\sta=(\star_T)_{\mbox{\it\tiny\texttt{a}}}\circ\wedge_{\{T\}}$ (where $(\star_T)_{\mbox{\it\tiny\texttt{a}}}$ is the \texttt{eab} semistar operation of finite type on $T$ associated to $\star_T$).
\end{oss}

 Olberding in \cite{ol, olb_noeth}, considered a new topology (called the {\it $\texttt{b}$-topology}) on the set $\overric(D)$ of integrally closed overrings of $D$ by taking, as a subbasis of open sets, the sets of the form 
 $$
 \mathscr{U}_{\mbox{\it\tiny\texttt{ic}}}(F, G):=\{T\in\overric(D)\mid F\subseteq G^{\mbox{\it\tiny \texttt{b}}(T)}\} ,
 $$
 where $F$ and $G$ range among the nonzero finitely generated $D$-submodules of $K$. He showed that the $\mbox{\it \texttt{b}}$-topology on $\overric(D)$ is finer than (or equal to) the Zariski topology (since it is straightforward that  $\texttt{B}_F =\mathscr{U}_{\mbox{\it\tiny\texttt{ic}}}(F,D)$, where $\texttt{B}_F:= \overric(D[F]) = \overric(D[x_1, x_2, \dots, x_n])$,  for each $F = x_1D+ x_2D+ \dots+x_n D\in \f(D)$, is subbasic open set of $\overric(D)$ with the  topology induced by the Zariski topology of $\overr(D)$) and the two topologies coincide when restricted to the Riemann-Zariski space $\Zar(D)$ \cite[Corollary 2.8]{ol}. 
 Using semistar operations, we can show more, i.e, the $\mbox{\it \texttt{b}}$-topology and the Zariski topology coincide on $\overric(D)$ (Corollary \ref{b-zar}).
 
\begin{prop} \label{overric}
Let $D$ be an integral domain, and consider the injective map
$\iota_{\texttt{ic,a}}: \overric(D) \rightarrow \SStarf(D)$ defined by $\iota_{\texttt{ic,a}}(T) := \texttt{b}(T)$, for each $T \in \overric(D) $. Assume that $\SStarf(D)$ is endowed with the Zariski topology. Then:
\begin{enumerate}[\rm(1)]
\item If $\overric(D)$ is endowed with the Zariski topology, then $\iota_{\texttt{ic,a}}$ is continuous and injective.
\item If $\overric(D)$ is endowed with the ${\mbox{\it \texttt{b}}}$-topology, then $\iota_{\texttt{ic,a}}$ is topological embedding, i.e., it estabilishes a homeomorphism between $\overric(D)$ and its image.
\end{enumerate}
\end{prop}
\begin{proof}
(1) Let $ \Psi_{\mbox{\it \tiny\texttt{a}}} := \Phi_{\mbox{\it \tiny\texttt{a}}}|_{\SStarf(D)}: \SStarf(D) \rightarrow \SStareabf(D)$,
 defined by $\Psi_{\mbox{\it \tiny\texttt{a}}}(\star) = \sta$, for each $\star \in \SStarf(D)$ (Lemma \ref{phia}). 
 Note that $\iota_{\mbox{\it \tiny\texttt{ic,a}}}=
 \Psi_{\mbox{\it \tiny\texttt{a}}} \circ \iota$,   since by  Remark \ref{rem:stara} we have  $\left(\wedge_{\{T\}}\right)_{\mbox{\it \tiny\texttt{a}}} 
 = \mbox{\it \texttt{b}}(T)$ for each overring $T$ of $D$ (see also the proof of Proposition \ref{iotaT}).
Therefore, $\iota_{\mbox{\it \tiny\texttt{ic,a}}}$
 is continuous as a composition of continuous maps.
 Moreover, if $T', T'' \in \overric(D)$ and $T' \neq T''$, then $\Zar(T') \neq \Zar(T'')$ and so 
 $\mbox{\it \texttt{b}}(T')\neq \mbox{\it \texttt{b}}(T'')$,
 i.e., $\iota_{\mbox{\it \tiny\texttt{ic,a}}}$ is injective.

(2)   Since the $\mbox{\it \texttt{b}}$-topology is finer than the Zariski topology, $\iota_{\mbox{\it \tiny\texttt{ic,a}}}$ is continuous by the previous part of the proof.  We set 
$$
\SStarb(D) := \{ \mbox{\it \texttt{b}}(T) \mid T \in\overric(D)\} = \iota_{\mbox{\it \tiny\texttt{ic,a}}}(\overric(D)) .
$$
 Let $\mathscr{U}_{\mbox{\it\tiny\texttt{ic}}}(F,G)$ be a subbasic open set of $\overric(D)$ in the ${\mbox{\it \texttt{b}}}$-topology.
 If $F:=x_1D+x_2D+\cdots+x_nD$, then $\mathscr{U}_{\mbox{\it\tiny\texttt{ic}}}(F,G)=\mathscr{U}_{\mbox{\it\tiny\texttt{ic}}}(x_1,G)\cap \mathscr{U}_{\mbox{\it\tiny\texttt{ic}}}(x_2,G)\cap\cdots\cap \mathscr{U}_{\mbox{\it\tiny\texttt{ic}}}(x_n,G)$, so we can suppose that $F=xD$ for some $0\neq x \in K$. 
In this situation we have:
$$
\iota_{\mbox{\it \tiny\texttt{ic,a}}}(\mathscr{U}_{\mbox{\it\tiny\texttt{ic}}}(xD,G))=\{\star \in \SStarb(D)\mid x\in G^\star\}=\texttt{V}_{x^{-1}G}\cap \SStarb(D) .
$$
Therefore, $\iota_{\mbox{\it \tiny\texttt{ic,a}}}$ is open onto $\SStarb(D)$ and hence it establishes a homeomorphism between $\overric(D)$ (with the ${\mbox{\it \texttt{b}}}$-topology) and $\SStarb(D)$ (with the Zariski topology). 
\end{proof}

\begin{cor} \label{b-zar}
The $\mbox{\it \texttt{b}}$-topology and the Zariski topology coincide on $\overric(D)$.
\end{cor}
\begin{proof} With the notation of the proof of Proposition \ref{overric}(2),
it is enough to observe that $\mathscr{U}_{\mbox{\it\tiny\texttt{ic}}}(xD,G)$ is also open in the Zariski topology of $\overric(D)$,
 since $\iota_{\mbox{\it \tiny\texttt{ic,a}}}^{-1}(\texttt{V}_{x^{-1}G}) = \mathscr{U}_{\mbox{\it\tiny\texttt{ic}}}(xD,G)$.
\end{proof}

The following can be seen as a companion of Proposition \ref{immersione:stab}.
\begin{prop}\label{iotaT}
Let $D$ be an integral domain and let $T \in \overr(D)$. The following properties are equivalent:
\begin{enumerate}[\rm(i)]
\item $\wedge_{\{T\}} \in \SStareabf(D)$;
\item $\iota(T)=\iota_{\texttt{ic,a}}(T)$;
\item $T$ is a Pr\"ufer domain.
\end{enumerate}
\end{prop} 
 \begin{proof}
By definition of the \texttt{eab} semistar operation of finite type associated to $\wedge_{\{T\}}$,  we have that $\wedge_{\{T\}}\in\SStareabf(D)$ if and only if $\wedge_{\{T\}}=(\wedge_{\{T\}})_{\mbox{\it\tiny \texttt{a}}}$. However, by Remark \ref{rem:stara} and Proposition \ref{overric}, and noting that $\wedge_{\{T\}}$ (when restricted to $\FF(T)$) coincides with the identity semistar $\mbox{\it \texttt{d}}_T$ on $\FF(T)$,  we have
\begin{equation*}
(\wedge_{\{T\}})_{\mbox{\it\tiny \texttt{a}}} =(\mbox{\it \texttt{d}}_T)_{\mbox{\it\tiny \texttt{a}}} \circ \wedge_{\{T\}} =\mbox{\it \texttt{b}}_T \circ \wedge_{\{T\}}= \mbox{\it \texttt{b}}(T)=\iota_{\mbox{\it\tiny \texttt{ic,a}}}(T),
\end{equation*}
where $\mbox{\it \texttt{b}}_T$   is the $\mbox{\it \texttt{b}}$-operation on $T$
and  $\mbox{\it \texttt{b}}(T) = \wedge_{\Zar(T)}$.
Therefore, (i) and (ii) are equivalent. It is obvious that (iii) $\Rightarrow$ (ii). For the reverse implication, it is enough to note that $\mbox{\it \texttt{d}}_T= \mbox{\it \texttt{b}}_T$ is equivalent to $T$ being Pr\"ufer \cite[Lemma 2]{fopi}. 
\end{proof}

Note that we can define the $\mbox{\it \texttt{b}}$-topology also on the whole $\overr(D)$. The pro\-perties that we obtain are somewhat similar to Propositions \ref{stable:subret}\ref{weak topology:stable} and \ref{stableft}\ref{weak topology:stableft}.

\begin{prop}\label{prop:bmap}
Let $D$ be an integral domain, and define the map
$
\bmap: \overr(D)  \rightarrow  \overric(D)$ by setting $\bmap(T) := T^{\mbox{\it\tiny \texttt{b}}} = \overline{T}$,
where $\overline{T}$ is the integral closure of $T$. 
\begin{enumerate}
\item[\rm(1)]\label{prop:bmap:zar} If  $\overric(D)$ and $\overr(D)$ are endowed with the Zariski topology, then $\bmap$ is continuous and, hence, it is a topological retraction of  $\overr(D)$ onto $\overric(D)$.
\item[\rm(2)]\label{prop:bmap:btop}  If  $\overric(D)$  is endowed with the Zariski topology and $\overr(D)$ is endowed with the $\mbox{\it\texttt{b}}$-topology, then $\bmap$
is the canonical  map onto the Kolmogoroff quotient space of $\overr(D)$.
\end{enumerate}
\end{prop}
\begin{proof}
(1) Let $x$ be a nonzero element of $K$ and let $\texttt{B}_x$ be a subbasic open set of the Zariski topology on $\overric(T)$. Then, 
\begin{equation*}
\begin{split}
\bmap^{-1}(\texttt{B}_x)& =\{T\in\overr(D)\mid x\in\overline{T}\}=\{T\in\overr(D)\mid x\text{~is integral over~}T\}=\\
& = \bigcup\left\{\texttt{B}_{\{\alpha_{k-1}, \alpha_{k-2},\ldots,\alpha_0\}}\mid x^k+\alpha_{k-1}x^{k-1}+\cdots+\alpha_1x+\alpha_0=0\right\}
\end{split}
\end{equation*}
which is open since it is a union of open sets. 
Hence, $\bmap$ is continuous. It is a  retraction since,  for $T \in \overric(D)$, $T^{\mbox{\it\tiny \texttt{b}}} =T$.

(2) Since the Zariski topology on $\overric(D)$   coincides with the $\mbox{\it\texttt{b}}$-topology, we can consider on $\overric(D)$ the subbasic open sets  $\mathscr{U}_{\mbox{\it\tiny \texttt{ic}}}(F, G)$, for $F,G \in \f(D)$.
 However, since $b(T)=b({\overline{T}})$, we have $\bmap^{-1}(\mathscr{U}_{\mbox{\it\tiny \texttt{ic}}}(F, G))=\mathscr{U}(F, G) := \{T\in\overr(D)\mid F\subseteq G^{\mbox{\it\tiny \texttt{b}}(T)}\}$, and thus $\bmap$ is continuous. 
 In the same way, $\bmap(\mathscr{U}(F, G))=\mathscr{U}_{\mbox{\it\tiny \texttt{ic}}}(F, G)$, so $\bmap$ is open. Finally, it is easy to see that two overrings $T'$ and $T''$ are topologically distinguishable by the ${\mbox{\it\texttt{b}}}$-topology if and only if they have the same integral closure; hence, $\bmap$ is the canonical  map onto the Kolmogoroff quotient space of $\overr(D)$
 with the ${\mbox{\it \texttt{b}}}$-topology.
\end{proof}

The relation between valutative operations and subsets of $\Zar(D)$ exhibits a si\-mi\-lar behaviour to the relation between spectral operations and subsets of $\spec(D)$.  Similarly to the prime spectrum case (see the paragraph preceding Proposition \ref{spettrali}),  for each subset $Y$ of the Riemann-Zariski space $\Zar(D)$, we set $Y^{\mbox{\rm\tiny\mathttt{gen}}}:= \{z \in \Zar(D) \mid y \in \chius(\{z\}) \mbox{ for some } y \in Y \}$ and we denote $\mbox{\rm{\texttt{Cl}}}^{\mbox{\rm\tiny\mathttt{inv}}}(Y)$ the closure of $Y$ in the inverse topology of  the spectral space $\Zar(D)$  (see \cite{ dofo-86}, \cite[Proposition 8]{ho}  and \cite{fifolo2, FiSp}).
Compare the next lemma with Proposition \ref{spettrali}.

\begin{lemma}\label{zar-inv}
	 Let $D$ be an integral domain and let $Y$ and $Z$ be two nonempty subsets of $\mbox{\rm\Zar}(D)$.
	Then, the following statements hold. 
	\begin{enumerate}[\rm(1)]
		\item $\wedge_Y=\wedge_Z$ if and only if $Y^{\mbox{\rm\tiny\mathttt{gen}}}=Z^{\mbox{\rm\tiny\mathttt{gen}}}$.

		\item $\wedge_Y$ is of finite type if and only if $Y$ is quasi-compact.

		\item $(\wedge_Y)_{\!{_f}}=(\wedge_Z)_{\!{_f}}$ if and only if 
		 $\mbox{\rm{\texttt{Cl}}}^{\mbox{\rm\tiny\mathttt{inv}}}(Y)=
		 \mbox{\rm{\texttt{Cl}}}^{\mbox{\rm\tiny\mathttt{inv}}}(Z)$.
		 \item $(\wedge_Y)_{\!{_f}}= \wedge_{\mbox{\rm\footnotesize{\texttt{Cl}}}^{\mbox{\rm\tiny\mathttt{inv}}}(Y)}$. 
	\end{enumerate}
\end{lemma}
\begin{proof}
(1) Note that, in the present situation,  $Y^{\mathttt{gen}} = \{ÊV \in \mbox{\rm\Zar}(D) \mid V \supseteq V_0, $ for some $V_0 \in Y \}$.
Assume first that $\wedge_Y=\wedge_Z$. Let $V$ be a valuation domain such that $V\in Y^{\mathttt{gen}}\setminus Z^{\mathttt{gen}}$. 
Then, for any $W\in Z$, we can pick an element $x_W\in W\setminus V$. 
It follows that $I:=(x_W^{-1} \mid W\in Z)\subseteq M_V$, where $M_V$ is the maximal ideal of $V$. 
Thus, if $V_0\in Y$ is such that $V_0\subseteq V$ (such a $V_0$ exists since $V\in Y^{\mathttt{gen}}$),
 we have $I V_0\subseteq M_{V_0}$ and, in particular, 
$1\notin I^{\wedge_Y}$. 
On the other hand, clearly $1\in I^{\wedge_Z}$, a contradiction. The converse it is straightforward since, for each $Y \subseteq 
\mbox{\rm\Zar}(D)$, $\wedge_Y = \wedge_{Y^{\mathttt{gen}}}$. 

For (2), (3) and (4), see respectively \cite[Proposition 4.5]{FiSp}, \cite[Theorem 4.9]{fifolo2} and, \cite[Corollary 4.17]{fifolo2}.
\end{proof}
\medskip

\begin{oss} \label{rk-2.4}
(a) Since $\mbox{\it \texttt{b}}=\wedge_{\Zar(D)}$ is a semistar operation of finite type on $D$ (and this can be proved completely independently from the topological point of view, see \cite[Proposition 6.8.2]{swhu} and \cite[Remark 4.6]{FiSp}), from Lemma \ref{zar-inv} we get a new proof of the fact that $\Zar(D)$ is a quasi-compact space (this is a special case of Zariski's theorem \cite[Theorem 40, page 113]{zs}).

 (b) Note that the equivalence (1) of Lemma \ref{zar-inv} can   also be viewed in a topological way, as indicated in Remark \ref{gen}, after replacing $\Zar(D)$  to $\spec(D)$.
 \end{oss}

 \medskip

Using the $\mbox{\it \texttt{b}}$-operation, Krull introduced a general version of the classical Kronecker fun\-ction ring, coinciding in case of Dedekind domains with the classical one (considered by L. Kronecker \cite{weyl, ed, gi}). 
In fact, a Kronecker fun\-ction ring can be defined starting by any \texttt{eab} semistar operation. 
In the next lemma, we summarize some properties of the Kronecker function ring, relevant to the remaining part of the paper.

\begin{lemma} \label{kr-star}
Let $D$ be an integral domain, $\star$ an \texttt{eab} semistar operation on $D$, $\X$ an indeterminate over $D$ and let $\boldsymbol{c}(h)$ be the content of a polynomial
 $h \in D[\X]$ (i.e., the ideal of $D$ generated by the coefficients of $h$).
 Set $\mathscr{V}(\star) := \{ V \in \mbox{\rm \Zar}(D) \mid F^\star \subseteq FV, \mbox{ for each } 
F \in \f(D)\}$ and 
$$
\begin{array}{rl}
\Kr(D, \star) :=& \hskip -7pt\{ f/g \mid   f,g \in D[\X],g 
\neq 0, 
\; 
\mbox{ 
	with } \boldsymbol{c}(f)
\subseteq \boldsymbol{c}(g)^\star  \} .
\end{array}
$$
Denote by $V(\X)$  the Gaussian (or trivial) extension of $V$ to $K(\X)$ (i.e., $V(\X)= V[\X]_{M[\X]}$, where $M$ is the maximal ideal of $V$).
 \begin{enumerate}
 \item[\rm(1)] $\Kr(D, \star)$ is a B\'ezout domain with quotient field $K(\X)$, called \emph{the $\star$-Kronecker function ring of $D$} and, for each polynomial $f \in D[\X]$, $\cont(f)\Kr(D, \star) = f\Kr(D, \star)$.
\item[\rm(2)] $\Kr(D, \star) = \bigcap\{ V(\X) \mid V \in {\mathscr{V}}(\star) \}$ and, for each $E \in \FF(D)$,
 $
 E^{\sta} = E \Kr(D, \star) \cap K = \bigcap \{ EV \mid V \in \mathscr{V}(\star) \} .
 $ 
 \item[\rm(3)] If $T$ is an overring of $D$, then ${\mathscr{V}}\left((\wedge_{\{T\}})_{{\mbox{\it \tiny \texttt{a}}}}\right) =
 {\mbox{\rm \Zar}}(T)$;  in particular, we reobtain that
 $(\wedge_{\{T\}})_{\mbox{\it \tiny \texttt{a}}} = \mbox{\it \texttt{b}}(T)$.
\end{enumerate}
 \end{lemma}
\begin{proof}
For (1), see \cite[Definition 3.2, Propositions 3.3 and 3.11(2), Corollary 3.4(2) and Theorem 5.1]{folo-2001}, \cite[Theorem 32.11]{gi} 
and \cite[Theorems 11 and 14]{folo-2006}. 
 For the proof of (2) see \cite[Proposition 4.1(5)]{folo-2001} and \cite[Theorem 14]{folo-2006}. 
 (3) is a direct consequence of (2) and of the definitions. 
\end{proof}
 
 \smallskip
 
 When considering Kronecker function rings, particularly important is the case $\star = {\mbox{\it\texttt{b}}}$ and, in this case, we simply set $\Kr(D):=\mbox{Kr}(D,\mbox{\it \texttt{b}})$. In this situation, it follows easily from the the fact that $\Kr(D)$ is a B\'ezout domain that
 the localization map $\spec(\Kr(D))\rightarrow \Zar(\Kr(D))$ (defined by $P\mapsto\Kr(D)_P$) is actually a homeomorphism. Moreover, the map $\Psi:\Zar(D)\rightarrow\Zar(\Kr(D))$ (defined by $V\mapsto V(\X)$) is a homeomorphism too \cite[Propositions 3.1 and 3.3]{fifolo2}, so that, by appropriate compositions, we deduce that there is a canonical homeomorphism between $\spec(\Kr(D))$ and $\Zar(D)$.

We are now in condition to prove the main result of this section.
 \begin{teor} \label{eab-op} 
	Let $D$ be an integral domain with quotient field $K$. Let $\Kr(D):=\Kr(D,\mbox{\it \texttt{b}})$
	 be the $\mbox{\it \texttt{b}}$-Kronecker function ring of $D$
	 and let ${\theta}: \spec(\Kr(D))\rightarrow \mbox{\rm \Zar}(D)$ be the homeomorphism defined by ${\theta}(Q) :=\Kr(D)_Q \cap K$, for each $Q \in \spec(\Kr(D))$ \cite[Theorem 2]{dofo-86}.
	Then, 
	\begin{enumerate}
	\item[\rm (1)]   The homeomorphism ${\theta}$ induces a   continuous  bijection $\boldsymbol{\Theta}: \SStarsp(\Kr(D)) \rightarrow $ $\SStarval(D) $ defined by setting $\boldsymbol{\Theta}(s_{ \mathscr{Y}}):=\wedge_{Z(\mathscr{Y})}$, for each $ \mathscr{Y} \subseteq \spec(\Kr(D))$, where $Z(\mathscr{Y}):= \{ V \in \mbox{\rm \Zar}(D) \mid M_V(\X) \cap \Kr(D) \in \mathscr{Y} \}$ and $M_V$ is the ma\-xi\-mal ideal of $V$.
	
 \item[\rm(2)] The map $\boldsymbol{\Theta}$, restricted to the semistar operations of finite type, gives rise to a homeomorphism 
 $\boldsymbol{\Theta}_{\mbox{\it \tiny\texttt{f}}}: \SStarspf(\Kr(D)) \rightarrow \SStareabf(D)$ of topological spaces (endowed with the Zariski topology).
 
	 \item[\rm(3)] $\SStareabf(D)$ is a spectral space.
	 \end{enumerate}
	 \end{teor}

\begin{proof}
(1) It is straightforward that the homeomorphism ${\theta}$ from $\spec(\Kr(D))$ to $\Zar(D)$  (which is, in particular, a isomorphism of partially ordered sets with the ordering induced by their topologies)  induces a 1-1 correspondence
 $\boldsymbol{\Theta}_{\!{_0}}$ between the set $\{ \mathscr{Y} \subseteq\spec(\Kr(D)) \mid \mathscr{Y} = \mathscr{Y}^\downarrow\}$ (where $ \mathscr{Y} ^\downarrow:= \{Q \in \spec(\Kr(D))\mid Q \subseteq Q',$ for some $Q' \in \mathscr{Y} \} = {\mathscr{Y}}^{\texttt{gen}}$) 
 and the set $\{ Y \subseteq \Zar(D) \mid Y= Y^{\uparrow}\} $ 
 (where $Y^{\uparrow}:= \{W \in \Zar(D) \mid W \supseteq W',$ for some $W' \in Y\} = {Y}^{\texttt{gen}}$). 
 Therefore $\boldsymbol{\Theta}_{\!{_0}}$ induces a bijection $\boldsymbol{\Theta}: \SStarsp(\Kr(D)) \rightarrow \SStarval(D) $ defined by $\boldsymbol{\Theta}(s_{ \mathscr{Y}}):=\wedge_{\boldsymbol{\Theta}_{\!{_0}}( \mathscr{Y} )}$, where ${\boldsymbol{\Theta}_{\!{_0}}( \mathscr{Y})} = \{ V \in \Zar(D) \mid M_V(\X) \cap \Kr(D) \in \mathscr{Y} \} =: Z(\mathscr{Y})$ (cf. Lemma \ref{zar-inv}, \cite[Corollaries 4.4 and 5.2, Proposition 5.1]{FiSp} and \cite[Lemma 4.2 and Remark 4.5]{fohu}).

Moreover, the map $\boldsymbol{\Theta}$ is continuous by \cite[Proposition 3.1(1)]{FiSp}, since
 $ (E\Kr(D))^{\mbox{\it\texttt{s}}_{ \mathscr{Y}}} $ $ \cap K 
 = E^{\wedge_{Z(\!\mathscr{Y})}}$ for each $E \in \FF(D)$ and for each $ \mathscr{Y} \subseteq\spec(\Kr(D))$. In fact,
$$
\begin{array}{rl}
E^{\wedge_{Z(\!\mathscr{Y})}}= & \hskip -7pt \bigcap\{EV \mid V \in Z(\!\mathscr{Y})\} = \bigcap \{EV\!(\X) \cap K \mid V \in Z(\!\mathscr{Y})\} \\
= & \hskip -7pt
E \left( \bigcap\{V(\X) \mid V \in Z(\!\mathscr{Y})\}\right) \cap K = E (\Kr(D))^{\mbox{\it\texttt{s}}_{\mathscr{Y}}} \cap K \\
= & \hskip -7pt (E \Kr(D))^{\mbox{\it\texttt{s}}_{\mathscr{Y}}} \cap K .
\end{array}
$$ 

  (2) Since quasi-compact sets correspond biunivocally to finite type semistar ope\-rations, in both spectral and the valutative case, then $\boldsymbol{\Theta}$ restricts to a continuous bijection $\boldsymbol{\Theta}_{\!{_{{f}}}}: \SStarspf(\Kr(D))\rightarrow \SStareabf(D)$ (see also \cite[Proposition 3.1(2)]{FiSp}). 

Let $J$ be a nonzero finitely generated ideal of $\Kr(D)$, thus it is principal (since $\Kr(D)$ is a B\'ezout domain). Therefore $J = z\Kr(D)$, for some nonzero element
 $ z:= \alpha/\beta \in K(\X)$, where $\alpha, \beta \in D[\X]$ and $\alpha, \beta$ are nonzero.
 We can consider the   basic open set  $\texttt{V}_{\alpha/\beta}^{\sharp} := \{\star \in \SStar(\Kr(D)) \mid 1 \in ((\alpha/\beta)\Kr(D))^\star\}$ of $\SStar(\Kr(D))$  (the superscript $\sharp$ is used here to emphasize the fact that we are considering subspaces of $\SStar(\Kr(D))$ and not of $\SStar(D)$).
 $$
\mbox{ {\bf Claim.}} \;\;\;\;\; \boldsymbol{\Theta}_{\mbox{\it \tiny\texttt{f}}} \left(\texttt{V}_{\alpha/\beta}^{\sharp}\cap\SStarspf(\Kr(D))\right)=
\left(\bigcap_{i=1}^n \texttt{V}_{b_i^{-1}\cont(\alpha)}\right)\cap \SStareabf(D) ,
$$
(where $\beta :=b_0+b_1\X\cdots+b_n\X^n$ and $\cont(\alpha)$ is the ideal of $D$ generated by the coefficients of the polynomial $\alpha$). 

Indeed, let $   \widetilde{\texttt{U}}_{\alpha/\beta}^{\sharp}:= \texttt{V}_{\alpha/\beta}^{\sharp}\cap\SStarspf(\Kr(D))$, suppose that $\star\in   {\widetilde{\texttt{U}}_{\alpha/\beta}^{\sharp}}$ and let $
\ast :=\boldsymbol{\Theta}_{\mbox{\it \tiny\texttt{f}}}(\star)$. 
Then, $1\in(\alpha/\beta \cdot \Kr(D))^\star$, i.e., 
$(\beta\Kr(D))^\star\subseteq(\alpha\Kr(D))^\star$. However, $\beta\Kr(D)=\cont(\beta)\Kr(D)$ (Lemma \ref{kr-star}(1)), 
and analogously for $\alpha$; thus,
\begin{equation*}
\cont(\beta)\subseteq(\cont(\beta)\Kr(D))^\star\cap K\subseteq(\cont(\alpha)\Kr(D))^\star\cap K=\cont(\alpha)^\ast.
\end{equation*}
Hence, $b_i\in\cont(\alpha)^\ast$ for each $i$, and $1\in(b_i^{-1}\cont(\alpha))^\ast$, that is, $\ast\in \texttt{V}_{b_i^{-1}\cont(\alpha)}$ for every $i$. On the other hand, it is a straightforward consequence of the definition that $\ast \in \SStareabf(D)$.

Conversely, let $\ast \in (\bigcap_{i=1}^n \texttt{V}_{b_i^{-1}\cont(\alpha)} )\cap \SStareabf(D)$. 
Since $\boldsymbol{\Theta}_{\mbox{\it \tiny\texttt{f}}} $ is bijective, then $\ast=\boldsymbol{\Theta}_{\mbox{\it \tiny\texttt{f}}}(\star) $ for a unique $\star \in \SStarspf(\Kr(D))$.
 Then, $b_i\in\cont(\alpha)^\ast$ for every $i$, 
 and $\cont(\beta)\subseteq\cont(\alpha)^\ast$; 
 it follows that $\beta/\alpha\in\Kr(D,\ast) $, i.e., $1 \in \alpha/\beta \cdot\Kr(D, \ast)$. On the other hand, $\ast=\boldsymbol{\Theta}_{\mbox{\it \tiny\texttt{f}}}(\star)$   implies 
 $E^\ast= (E\Kr(D))^\star\cap K $ and, moreover, since $\ast$ is \texttt{eab}, $E^\ast = E\Kr(D,\ast) \cap K$, for each $E \in \FF(D)$ (Lemma \ref{kr-star}(2)). 
 Therefore, $1\in \alpha/\beta \cdot \Kr(D,\ast)$ implies that $1\in (\alpha/\beta \cdot\Kr(D))^\star$, i.e., 
 $\star\in {\widetilde{\texttt{U}}_{\alpha/\beta}^{\sharp}}$, so that $\ast\in \boldsymbol{\Theta}_{\mbox{\it \tiny\texttt{f}}}({\widetilde{\texttt{U}}_{\alpha/\beta}^{\sharp}})$.

The claim ensures that $\boldsymbol{\Theta}_{\mbox{\it \tiny\texttt{f}}}$ is open, and hence we conclude that it is a homeomorphism. 

(3) is an easy consequence of (2) and Theorem \ref{stable spectral}.
\end{proof}

 \begin{oss}
As we did in Remark \ref{oss:spectral-constr}, we can ask if $\SStareabf(D)$ is closed when $\SStarf(D)$ is endowed with the constructible topology.
 The answer is positive; indeed, consider the map $\boldsymbol{\Lambda}:\SStarstabft(\Kr(D))\rightarrow\SStarf(D)$ obtained by composing $\boldsymbol{\Theta}_f$ with the inclusion of $\SStareabf(D)$ into $\SStarf(D)$. Let $I \in \f(D)$ and let  $\mbox{\rm\texttt{V}}_I$ be a subbasic open set of $\SStar(D)$ then, by the proof of Theorem \ref{eab-op}, there exist  $\alpha, \beta \in D[\X]$, with $\alpha, \beta$  nonzero, such that 
\begin{equation*}
\boldsymbol{\Lambda}^{-1}(\mbox{\rm\texttt{V}}_I \cap \SStarf(D))=\mbox{\rm\texttt{V}}_{\alpha/\beta}^{\sharp}\cap\SStarstabft(\Kr(D))=  \widetilde{\mbox{\rm\texttt{{U}}}}^{\sharp}_{((\alpha/\beta)\Kr(D))\cap\Kr(D)}.
\end{equation*}
Since $(({\alpha}/{\beta})\Kr(D))\cap\Kr(D)$ is finitely generated ($\Kr(D)$ being a Pr\"ufer domain \cite[(25.4), part (1)]{gi}), $\boldsymbol{\Lambda}^{-1}(\mbox{\rm\texttt{V}}_I \cap \SStarf(D))$ is quasi-compact   (Proposition \ref{stableft}(1)). This means that $\boldsymbol{\Lambda}$ is a spectral map. In particular, it is continuous when $\SStarstabft(\Kr(D))$ and $\SStarf(D)$ are endowed with the constructible topology. Since a spectral space, endowed with the constructible topology, is both compact and Haussdorff, $\boldsymbol{\Lambda}$ is a closed map, when both spaces are endowed with the constructible topology. In particular, $\boldsymbol{\Lambda}(\SStarstabft(\Kr(D)))=\SStareabf(D)$ is closed subspace of $\SStarf(D)$, endowed with the constructible topology.
\end{oss}
\smallskip

\noindent  
{\bf Acknowledgment.} We sincerely thank the referee for the careful reading of the manuscript  and for providing constructive comments and help in improving the presentation of the paper.


\medskip


\end{document}